\documentclass[10pt]{amsart}

\usepackage{amssymb,amsmath,amsthm}
\usepackage[arc,all]{xy}
\usepackage{eucal}

\theoremstyle{plain}

\newtheorem{lem}[subsection]{Lemma}
\newtheorem{prop}[subsection]{Proposition}
\newtheorem{thm}[subsection]{Theorem}

\newtheorem{assumption}[subsection]{Assumption}

\theoremstyle{definition}
\newtheorem{defn}[subsection]{Definition}

\theoremstyle{remark}
\newtheorem{rem}[subsection]{Remark}

\newcommand{\ZZ}{{ \mathbb{Z} }}

\newcommand{\Mod}{{ \mathsf{Mod} }}
\newcommand{\ModR}{{ \mathsf{Mod}_\capR }}

\newcommand{\Alg}{{ \mathsf{Alg} }}

\newcommand{\TQ}{{ \mathsf{TQ} }}

\newcommand{\AlgJ}{{ \Alg_J }}

\newcommand{\AlgO}{{ \Alg_\capO }}

\newcommand{\capO}{{ \mathcal{O} }}

\newcommand{\capR}{{ \mathcal{R} }}
\newcommand{\capX}{{ \mathcal{X} }}

\newcommand{\id}{{ \mathrm{id} }}

\newcommand{\wequiv}{{ \ \simeq \ }}
\newcommand{\iso}{{ \cong }}

\newcommand{\function}[3]{{ {#1}\colon\thinspace{#2}\rightarrow{#3} }}

\DeclareMathOperator*{\hocolim}{hocolim}

\DeclareMathOperator*{\holim}{holim}

\DeclareMathOperator{\Tot}{Tot}

\DeclareMathOperator{\hofib}{hofib}

\newcommand{\abs}[1]{\left\vert#1\right\vert}
\newcommand{\sett}[1]{\left\{#1\right\}}

\title[Fibration theorems for $\TQ$-completion]{Fibration theorems for $\TQ$-completion of structured ring spectra}

\author{Nikolas Schonsheck}

\address{Department of Mathematics, The Ohio State University, 231 West 18th Ave, Columbus, OH 43210, USA}

\email{schonsheck.2@osu.edu}

\begin{document}

\maketitle

\begin{abstract}
The aim of this short paper is to establish a spectral algebra analog of the Bousfield-Kan ``fibration lemma'' under appropriate conditions. We work in the context of algebraic structures that can be described as algebras over an operad $\capO$ in symmetric spectra. Our main result is that completion with respect to topological Quillen homology (or $\TQ$-completion, for short) preserves homotopy fibration sequences provided that the base and total $\capO$-algebras are connected. Our argument essentially boils down to proving that the natural map from the homotopy fiber to its $\TQ$-completion tower is a pro-$\pi_*$ isomorphism. More generally, we also show that similar results remain true if we replace ``homotopy fibration sequence'' with ``homotopy pullback square.''
\end{abstract}

\section{Introduction}
This paper is written in the context of symmetric spectra \cite{Hovey_Shipley_Smith, Schwede_book_project}, and more generally, modules over a commutative ring spectrum $\capR$; see \cite{EKMM} for another approach to a well-behaved category of spectra. We consider any algebraic structure in the closed symmetric monoidal category $(\ModR,\wedge,\capR)$ of $\capR$-modules that can be described as algebras over a reduced operad $\capO$; that is, $\capO[0]=*$ is the trivial $\capR$-module and hence $\capO$-algebras are non-unital (see, for instance, \cite{Ching_Harper_derived_Koszul_duality, Harper_Hess}).

Topological Quillen homology, or $\TQ$-homology for short, is the precise $\capO$-algebra analog of ordinary homology for spaces and is weakly equivalent to the stabilization of $\capO$-algebras; see, for instance,  \cite{Basterra, Basterra_Mandell, Harper_Hess, Lawson}. The $\TQ$-completion  of an $\capO$-algebra $X$, denoted $X^\wedge_\TQ$, is supposed to be the ``part of the $\capO$-algebra that $\TQ$-homology sees'' (\cite{Ching_Harper_derived_Koszul_duality, Harper_bar_constructions, Harper_Zhang}). Analogous to Bousfield-Kan's $\mathbb{Z}$-completion \cite{Bousfield_Kan} of a space, $X_\TQ^\wedge$ is the homotopy limit of the cosimplicial resolution built by iterating the unit map of the monad associated to the $\TQ$-homology adjunction. We review these constructions in Section \ref{sec:Background_on_TQ}, but to keep this paper appropriately concise, we freely use the notation in \cite{Ching_Harper_derived_Koszul_duality}.

Suppose we start with a fibration sequence  $F\rightarrow E\rightarrow B$ of $\capO$-algebras and consider the associated commutative diagram of the form
\begin{align}
\label{eq:fibration_sequence_diagram}
\xymatrix{
  F\ar[d]^{(*)}\ar[r] & 
  E\ar[d]\ar[r] & 
  B\ar[d]\\
  F^\wedge_\TQ\ar[r] & 
  E^\wedge_\TQ\ar[r] & 
  B^\wedge_\TQ
}
\end{align}
in $\AlgO$. The aim of this short paper is to establish sufficient conditions on $E\rightarrow B$ such that the bottom row is also a fibration sequence. In other words, we are interested in establishing a $\TQ$-completion analog of the Bousfield-Kan ``fibration lemma'' \cite[II.2.2]{Bousfield_Kan}, under appropriate additional conditions on $E\rightarrow B$. If we are in the special situation where $E,B$ are $\TQ$-complete (i.e., their  coaugmentation maps in  \eqref{eq:fibration_sequence_diagram} are weak equivalences), then this amounts to verifying that $(*)$ is a weak equivalence. The following theorem is our main result.

\begin{thm}[$\TQ$-completion fibration theorem]
\label{MainTheorem1}
Let $E\rightarrow B$ be a fibration of $\capO$-algebras with fiber $F$. If $E,B$ are $0$-connected, then the $\TQ$-completion map $F\wequiv F^\wedge_\TQ$ is a weak equivalence; furthermore, the natural map from $F$ to its $\TQ$-completion tower is a pro-$\pi_*$ isomorphism.
\end{thm}

This idea generalizes. Suppose we instead start with a fibration $\function{p}{X}{Y}$ that fits into a left-hand pullback square of the form
\begin{align}
\xymatrix{
A \ar[r] \ar[d] & X \ar[d]^{p}\\
B \ar[r] & Y	
}\quad\quad\quad
\xymatrix{
A^\wedge_\TQ \ar[r] \ar[d] & X^\wedge_\TQ \ar[d]\\
B^\wedge_\TQ \ar[r] & Y^\wedge_\TQ
}
\end{align}
in $\AlgO$. We would like to establish sufficient conditions on the pullback data $B\rightarrow Y\leftarrow X$ such that the right-hand square of the indicated form is also a homotopy pullback diagram. Similar to above, if $B,X,Y$ are $\TQ$-complete, then this amounts to verifying that the $\TQ$-completion map $A\wequiv A^\wedge_\TQ$ is a weak equivalence. The following theorem is a generalization of our main result.

\begin{thm}[$\TQ$-completion homotopy pullback theorem]
\label{MainTheorem2}
Consider any pullback square of the form
\begin{align}
\label{eq:pullback_diagram_generic}
\xymatrix{
A \ar[r] \ar[d] & X \ar[d]^{p}\\
B \ar[r] & Y	
}
\end{align}
in $\AlgO$, where $p$ is a fibration. If $B,X,Y$ are $0$-connected, then the $\TQ$-completion map $A \simeq A_\TQ^\wedge$ is a weak equivalence; furthermore, the natural map from $A$ to its $\TQ$-completion tower is a pro-$\pi_*$ isomorphism.
\end{thm}

\begin{rem}
It is probably worth pointing out that our strategy of attack works with $\capO$-algebras replaced by pointed simplicial sets. In more detail: Consider any pullback diagram of the form \eqref{eq:pullback_diagram_generic} in pointed simplicial sets, where $p$ is a fibration, and assume that $A$ is connected. If $B,X,Y$ are 1-connected, then the Bousfield-Kan $\ZZ$-completion map $A \simeq A_\ZZ^\wedge$ is a weak equivalence; furthermore, the natural map from $A$ to its $\ZZ$-completion tower is a pro-$\pi_*$ isomorphism. This provides a new proof of the result in Bousfield-Kan (see, for instance \cite[III.5.3]{Bousfield_Kan}) that such $A$ are $\ZZ$-complete.
\end{rem}

For technical reasons explained in Remark \ref{rem:reason_for_assumption}, we make the following assumption; for instance, it allows for an iterable point-set model of $\TQ$-homology and hence an associated point-set model for the $\TQ$-resolution of a cofibrant $\capO$-algebra. Note that to say an operad $\capO$ is $n$-connected means that, for each $r \geq 0$, its constituent $\capO[r]$ is $n$-connected.

\begin{assumption}\label{assumption}
	Throughout this paper, $\capO$ denotes a reduced operad in the closed symmetric monoidal category $(\ModR,\wedge,\capR)$ of $\capR$-modules (see, for instance, \cite{Hovey_Shipley_Smith, Schwede_book_project, Shipley_commutative_ring_spectra}). We assume that $\capO,\capR$ are $(-1)$-connected, and that $\capO$ satisfies the following cofibrancy condition: Consider the unit map $I\rightarrow\capO$; we assume that $I[r]\rightarrow\capO[r]$ is a flat stable cofibration (\cite[7.7]{Harper_Hess}) between flat stable cofibrant objects in $\ModR$ for each $r \geq 0$. This is exactly the cofibrancy condition appearing in \cite[2.1]{Ching_Harper_derived_Koszul_duality}. Unless stated otherwise, we work in the positive flat stable model structure \cite{Harper_Hess} on $\AlgO$.
\end{assumption}

\noindent
{\bf Relationship to previous work.} Ching-Harper prove in \cite{Ching_Harper_derived_Koszul_duality} that all $0$-connected $\capO$-algebras are $\TQ$-complete. However, it was known that this class does not represent all $\TQ$-complete $\capO$-algebras; for instance, one can show that any $\capO$-algebra in the image of $U$ (see Section \ref{sec:Background_on_TQ}) is $\TQ$-complete, by an extra codegeneracy argument. We conjecture that any $\capO$-algebra with a principally refined Postnikov tower is $\TQ$-complete, which would mirror the analogous result for $\mathbb{Z}$-completion of spaces. This paper is a first step in that direction.\\

\noindent
{\bf Acknowledgments.} The author wishes to thank John E. Harper for his support and advice, Yu Zhang for many helpful conversations, and Jake Blomquist, Duncan Clark, and Sarah Klanderman for useful discussions. The author would also like to thank an anonymous referee for helpful comments on improving the exposition and clarity of the paper. The author was supported in part by National Science Foundation grant DMS-1547357 and the Simons Foundation: Collaboration Grants for Mathematicians \#638247.

\section{Outline of the main argument}

We will now outline the proof of Theorem \ref{MainTheorem1}. It suffices to consider the case of a fibration $E\rightarrow B$ in $\AlgO$ between cofibrant objects (otherwise, cofibrantly replace). The first step is (i) to build the associated cosimplicial resolutions of $E,B$ with respect to $\TQ$-homology by iterating the $\TQ$-Hurewicz map $\id\rightarrow UQ$ (see Section \ref{sec:Background_on_TQ}) and (ii) to construct the coaugmented cosimplicial diagram $F\rightarrow \tilde{F}$ that is built by taking (functorial) homotopy fibers vertically, followed by objectwise (functorial) cofibrant replacements. In this way, we obtain a commutative diagram of the form 
\begin{align}
\label{eq:levelwise_fiber_sequences}
\xymatrix{
F \ar[d] \ar[r] & 
\tilde{F}^0 \ar[d] \ar@<0.5ex>[r] \ar@<-0.5ex>[r] & 
\tilde{F}^1 \ar[d] \ar@<1ex>[r] \ar[r] \ar@<-1ex>[r] & 
\tilde{F}^2 \cdots \ar@<-2ex>[d]\\
E \ar[d] \ar[r] & 
(UQ)E \ar[d] \ar@<0.5ex>[r] \ar@<-0.5ex>[r] & 
(UQ)^2 E \ar[d] \ar@<1ex>[r] \ar[r] \ar@<-1ex>[r] & 
(UQ)^3E \cdots \ar@<-2ex>[d]\\	
B \ar[r] & 
(UQ)B \ar@<0.5ex>[r] \ar@<-0.5ex>[r] & 
(UQ)^2B \ar@<1ex>[r] \ar[r] \ar@<-1ex>[r]& 
(UQ)^3B \cdots
}
\end{align}
in $\AlgO$, where the vertical columns are homotopy fibration sequences. Replacing if needed, we may also assume that $\tilde{F}$ is a Reedy fibrant cosimplicial $\capO$-algebra.

\begin{rem}
For ease of notational purposes, we usually suppress the codegeneracy maps in $\Delta$-shaped diagrams appearing throughout this paper.
\end{rem}

 Applying $\holim_\Delta$ (see \cite[Section 8]{Ching_Harper_derived_Koszul_duality}) to the maps of $\Delta$-shaped diagrams in \eqref{eq:levelwise_fiber_sequences}, where we regard the left-hand vertical column as maps of constant $\Delta$-shaped diagrams, gives a commutative diagram in $\AlgO$ of the form
\begin{align*}
\xymatrix{
  F\ar[r]\ar[d] & E\ar[r]\ar[d]^-{\simeq} & B\ar[d]^-{\simeq}\\
  \holim_\Delta\tilde{F}\ar[r] &
  E^\wedge_\TQ\ar[r] &
  B^\wedge_\TQ
}
\end{align*}
where each row is a homotopy fibration sequence. The indicated maps are weak equivalences by \cite[1.6]{Ching_Harper_derived_Koszul_duality} since $E,B$ are assumed to be 0-connected. It follows that the left-hand map $F \to \holim_\Delta\tilde{F}$ is a weak equivalence as well.

The next step is to get the $\TQ$-completion of $F$ into the picture; the basic idea is to prove that $F \to \holim_\Delta\tilde{F}$ is weakly equivalent to the natural coaugmentation $F\rightarrow F^\wedge_\TQ$.  Our strategy of attack is to objectwise resolve, with respect to $\TQ$-homology, the upper horizontal diagram $F\rightarrow\tilde{F}$ 

\begin{align}
\label{eq:resolution_of_F_diagram}
\xymatrix{
\overset{\raisebox{0.5ex}{\vdots}}{(UQ)^3 F} \ar[r]^-{(\#)} & \overset{\raisebox{0.5ex}{\vdots}}{(UQ)^3\tilde{F}^0} \ar@<0.5ex>[r] \ar@<-0.5ex>[r] & \overset{\raisebox{0.5ex}{\vdots}}{(UQ)^3\tilde{F}^1} \ar@<1ex>[r] \ar[r] \ar@<-1ex>[r] & \overset{\raisebox{0.5ex}{\vdots}}{(UQ)^3\tilde{F}^2} \cdots\\
(UQ)^2 F \ar@<1ex>[u] \ar[u] \ar@<-1ex>[u] \ar[r]^-{(\#)} & (UQ)^2\tilde{F}^0 \ar@<1ex>[u] \ar[u] \ar@<-1ex>[u] \ar@<0.5ex>[r] \ar@<-0.5ex>[r] & (UQ)^2\tilde{F}^1 \ar@<1ex>[u] \ar[u] \ar@<-1ex>[u] \ar@<1ex>[r] \ar[r] \ar@<-1ex>[r] & (UQ)^2\tilde{F}^2 \ar@<2.5ex>[u] \ar@<1.5ex>[u] \ar@<0.5ex>[u] \cdots\\
(UQ)F \ar@<0.5ex>[u] \ar@<-0.5ex>[u] \ar[r]^-{(\#)} & 
(UQ)\tilde{F}^0 \ar@<0.5ex>[u] \ar@<-0.5ex>[u] \ar@<0.5ex>[r] \ar@<-0.5ex>[r] & 
(UQ)\tilde{F}^1\ar@<0.5ex>[u] \ar@<-0.5ex>[u] \ar@<1ex>[r] \ar[r] \ar@<-1ex>[r] & 
(UQ)\tilde{F}^2\ar@<2ex>[u] \ar@<1ex>[u] \cdots\\
F \ar[u]^-{(*)'} \ar[r]^-{(\#)} & \tilde{F}^0 \ar[u]_-{(**)} \ar@<0.5ex>[r] \ar@<-0.5ex>[r] & 
\tilde{F}^1 \ar[u]_-{(**)} \ar@<1ex>[r] \ar[r] \ar@<-1ex>[r]& \tilde{F}^2 \ar@<1.5ex>[u]_-{(**)} \cdots
}
\end{align}
in \eqref{eq:levelwise_fiber_sequences}, and show that the maps $(\#)$ and $(**)$ induce weak equivalences after applying $\holim_\Delta$ (Propositions \ref{prop:horizontal_direction} and \ref{prop:vertical_direction}). Once this has been accomplished, we obtain a commutative diagram of the form
\begin{align}
\xymatrix{
\holim_\Delta(UQ)^{\bullet+1}F \ar[r]^-\simeq & \holim_{\Delta\times\Delta}(UQ)^{\bullet+1}\tilde{F}\\
F \ar[u]^-{(*)'} \ar[r]^\simeq & \holim_\Delta\tilde{F} \ar[u]^-\simeq
}
\end{align}
and conclude that the natural coagumentation map $F \simeq F_\TQ^\wedge \simeq \holim_\Delta(UQ)^{\bullet+1}F$ is a weak equivalence. This proves the first part of Theorem \ref{MainTheorem1}.

The second part of Theorem \ref{MainTheorem1} requires additional work. In order to precisely formulate this stronger result, we introduce the following two definitions.

\begin{defn}
	A map of towers of $\capO$-algebras $\sett{X_s}_s \to \sett{Y_s}_s$ is a \emph{pro-$\pi_\ast$ isomorphism} if the induced map
	\[
	\sett{\pi_nX_s}_s \to \sett{\pi_nY_s}_s
	\]
	of (abelian) groups towers is a pro-isomorphism for each $n \in \mathbb{Z}$. (Throughout this paper, we assume all homotopy groups are derived \cite{Schwede_homotopy_groups, Schwede_book_project}.)
\end{defn}

\begin{rem}\label{rem:pro_pi_iso_implies_weak_equivalence}
	Given a pro-$\pi_\ast$ isomorphism as above, it follows from the associated $\lim^1$ short exact sequence that the induced map $\holim_sX_s \simeq \holim_sY_s$ is a weak equivalence; see, for instance, \cite[Section 8]{Ching_Harper_derived_Koszul_duality}.
\end{rem}

\begin{defn}
	Define $\mathsf{Tot}$ as the right derived functor of $\Tot$ in $(\AlgO)^\Delta$ equipped with the Reedy model structure. In other words, given a cosimplicial $\capO$-algebra $X$, we define $\mathsf{Tot}(X)$ to be $\Tot(RX)$ where $RX$ is the (functorial) Reedy fibrant replacement of $X$ in $\AlgO^\Delta$. 
\end{defn}

Stated precisely, the second part of Theorem \ref{MainTheorem1} asserts that the map of towers
\begin{align*}
\sett{F}_s \overset{(*)'}{\longrightarrow} \sett{\mathsf{Tot}_s(UQ)^{\bullet+1}F}_s
\end{align*}
is a pro-$\pi_\ast$ isomorphism. To show that this assertion is true, note that the proofs of Propositions \ref{prop:horizontal_direction} and \ref{prop:vertical_direction} imply that the tower maps 
\[
\sett{(UQ)^kF}_s \to \sett{\mathsf{Tot}_s(UQ)^k\tilde{F}}_s \text { and } \sett{\tilde{F}^n}_s \to \sett{\mathsf{Tot}_s(UQ)^{\bullet+1}\tilde{F}^n}_s
\]
are actually pro-$\pi_\ast$ isomorphisms for each $k, n \geq 0$. Now consider the commutative diagram of towers of the form
\begin{align}
\xymatrix{
\sett{\mathsf{Tot}_s(UQ)^{\bullet+1}}_s \ar[r] & \sett{\mathsf{Tot}_s\mathsf{Tot}_s(UQ)^{\bullet+1}\tilde{F}}_s\\
\sett{F}_s \ar[u]^-{(*)'} \ar[r] & \sett{\mathsf{Tot}_s\tilde{F}}_s \ar[u]
}
\end{align}
It follows from the tower lemma below that the horizontal and right-hand vertical maps are pro-$\pi_\ast$ isomorphisms, and hence the map $(*)'$ is as well.
\begin{prop}[Tower lemma for $\capO$-algebras]\label{prop:tower_lemma}
	Suppose we are given a map from the Reedy fibrant cosimplicial $\capO$-algebra $X$ to a tower of Reedy fibrant cosimplicial $\capO$-algebras $\sett{Y_s}_s$
	\begin{align*}
	\xymatrix{
	X^0 \ar@<0.5ex>[r] \ar@<-0.5ex>[r] \ar[d] & X^1 \ar@<1ex>[r] \ar[r] \ar@<-1ex>[r] \ar[d] & X^2 \ar@<-2ex>[d] \cdots\\	
	\sett{Y^0_s}_s \ar@<0.5ex>[r] \ar@<-0.5ex>[r] & \sett{Y^1_s}_s \ar@<1ex>[r] \ar[r] \ar@<-1ex>[r] & \sett{Y^2_s}_s \cdots
	}
	\end{align*}
	If $\sett{X^k}_s \to \sett{Y^k_s}_s$ induces a pro-$\pi_\ast$ isomoprhism for each fixed $k$, then 
	\[
	\sett{\Tot_nX}_s \to \sett{\Tot_nY_s}_s
	\]
	induces a pro-$\pi_\ast$ isomorhpism for each fixed $n$.
\end{prop}
\begin{proof}
In the context of spaces, this is proven in \cite[1.4]{Dwyer_exotic_convergence} and the same argument remains valid in our setting.
\end{proof}

\section{Background on $\TQ$-homology and $\TQ$-completion}\label{sec:Background_on_TQ}
The purpose of this section is to briefly recall the definition of $\TQ$-homology and its associated completion construction. For a more thorough introduction, useful references include \cite{Ching_Harper_derived_Koszul_duality},  \cite{Harper_bar_constructions}, and \cite{Harper_Hess}.

\begin{defn}
	Given an operad $\capO$, its \emph{1-truncation $\tau_1\capO$} is the operad defined by
	\begin{align*}(\tau_1\capO)[r] := 
	\begin{cases}
		\capO[r],& \text{for } r \leq 1,\\
		\hfill \ast,& \text{otherwise}
	\end{cases} 
	\end{align*}
\end{defn}
The canonical map of operads $\capO \to \tau_1\capO$ induces the following change of operads adjunction, with left adjoint on top.
\begin{align}\label{eq:barq_baru}
\xymatrix{
\AlgO \ar@<0.5ex>^-{\bar{Q}}[r] & \Alg_{\tau_1\capO} \ \iso \ \Mod_{\capO[1]} \ar@<0.5ex>^-{\bar{U}}[l]
}
\end{align}

Here, $\bar{Q}(X):= \tau_1\capO\circ_\capO(X)$ and $\bar{U}$ is the forgetful functor. It is proven in \cite{Harper_symmetric_spectra} and \cite{Harper_Hess} that this is, in fact, a Quillen adjunction.

\begin{defn}
	Let $X$ be an $\capO$-algebra. The \emph{$\TQ$-homology} of $X$ is the $\capO$-algebra $\TQ(X):= \mathsf{R}\bar{U}\left(\mathsf{L}\bar{Q}(X)\right)$, where $\mathsf{L}$ and $\mathsf{R}$ indicate the appropriate derived functors.
\end{defn}

We would then like to form a cosimplicial (or Godement) resolution of the form

\begin{align}\label{eq:tq_resolution}
\xymatrix{
X \ar[r] & (\TQ) X \ar@<0.5ex>[r] \ar@<-0.5ex>[r] & (\TQ)^2X \ar@<1ex>[r] \ar[r] \ar@<-1ex>[r] & (\TQ)^3X \cdots	
}
\end{align}

Although $\TQ(X) \simeq \bar{U}\bar{Q}(X)$ if $X$ is cofibrant, the forgetful functor $\bar{U}$ need not send cofibrant objects in $\Alg_{\tau_1\capO}$ to cofibrant objects in $\AlgO$. Consequently, there is no guarantee that $(\TQ)^nX \simeq (\bar{U}\bar{Q})^nX$ for $n \geq 2$. The canonical cosimplicial resolution associated to \eqref{eq:barq_baru} is therefore unlikely to be of the form $\eqref{eq:tq_resolution}$.

Because of this difficulty, an additional maneuver is required to construct an iterable point-set model for $\TQ(X)$. We follow \cite[3.16]{Harper_Hess} to produce a rigidified version of \eqref{eq:tq_resolution}. First, factor the operad map $\capO \to \tau_1\capO$ as
\begin{align*}
\capO \to J \to \tau_1\capO
\end{align*}
a cofibration followed by a weak equivalence. This induces (Quillen) adjunctions
\begin{align}\label{eq:O_J_tau_adjunctions}
\xymatrix{
	\AlgO \ar^-{Q}@<.5ex>[r] & \AlgJ \ar@<0.5ex>[r] \ar^-{U}@<0.5ex>[l] & \Alg_{\tau_1\capO} \ar@<0.5ex>[l]
}
\end{align}
where $Q(X) := J\circ_\capO(X)$ and $U$ is the forgetful functor.
\begin{rem}
The adjunction on the right is, in fact, a Quillen equivalence (see \cite[7.21]{Harper_Hess}) and we therefore think of $\AlgJ$ as a ``fattened up'' version of $\Alg_{\tau_1\capO}$. Furthermore, it follows that $\TQ(X) \simeq UQ(X)$ if $X$ is a cofibrant $\capO$-algebra.
\end{rem}

The advantage of \eqref{eq:O_J_tau_adjunctions} is that the forgetful functor $U$ sends cofibrant objects in $\AlgJ$ to cofibrant objects in $\AlgO$ (see \cite[5.49]{Harper_Hess}). For a cofibrant $\capO$-algebra $X$, we therefore have weak equivalences of the form $(\TQ)^nX \simeq (UQ)^nX$ for all $n \geq 1$. Hence, the canonical resolution
\begin{align}
\xymatrix{
X \ar[r] & (UQ)^{\bullet+1}X \colon (UQ)X \ar@<0.5ex>[r] \ar@<-0.5ex>[r] & (UQ)^2X \ar@<1ex>[r] \ar[r] \ar@<-1ex>[r] & (UQ)^3X \cdots	
}
\end{align}
is of the desired form \eqref{eq:tq_resolution}.

\begin{defn}
	Let $X$ be an $\capO$-algebra. The \emph{$\TQ$-completion of $X$} is the $\capO$-algebra $X_\TQ^\wedge := \holim_\Delta(UQ)^{\bullet+1}(X^c)$, where $X^c$ denotes the functorial cofibrant replacement of $X$ in $\AlgO$.
\end{defn}

\begin{rem}\label{rem:reason_for_assumption}
	We are now in a better position to explain the reasons for Assumption \ref{assumption}, which are as follows. The connectivity assumption on $\capO$ and $\mathcal{R}$ guarantees the results of \cite{Ching_Harper} used below are applicable, while the cofibrancy condition on $\capO$ ensures \cite[5.49]{Harper_Hess} that the forgetful functor $\AlgJ \to \AlgO$ preserves cofibrant objects.
\end{rem}

\section{Analysis of the horizontal direction}\label{sec:horizontal_direction}
The purpose of this section is to analyze the maps $(\#)$ and, in particular, to prove Proposition \ref{prop:horizontal_direction}. The basic idea is to, first, establish uniform cartesian estimates on the canonical coface cubes (see Section \ref{sec:appendix}) associated to the coaugmented cosimplicial $\capO$-algebra $F \to \tilde{F}$. This is the content of Proposition \ref{prop:uniform_cartesian_estimates_for_F} and is accomplished by analyzing the corresponding coface cubes of $E \to (UQ)^{\bullet+1}E$ and $B \to (UQ)^{\bullet+1}B$. We next show, in Proposition \ref{prop:UQ_preserves_id_cartesian}, that objectwise application of the $\TQ$-homology spectrum functor preserves this cartesian estimate. Proposition \ref{prop:horizontal_direction} then follows inductively. 

Our analysis in this section will involve a number of concepts from cubical homotopy theory. We provide an overview of the relevant details in Section \ref{sec:appendix}.

\begin{defn}\label{def:coface_cubes}
	Let $\mathcal{E}_{n+1}$ be the coface $(n+1)$-cube associated to the coaugmented cosimplicial $\capO$-algebra $E \to (UQ)^{\bullet+1}E$ and define $\mathcal{B}_{n+1}$ similarly. Let $\tilde{\mathcal{F}}_{n+1}$ be the coface $(n+1)$-cube associated to the coaugmented cosimplicial $\capO$-algebra $F \to \tilde{F}$.
\end{defn}

The following proposition gives the uniform cartesian estimates on $\mathcal{E}_{n+1}$ and $\mathcal{B}_{n+1}$ (by setting $k = 0$) that we will ultimately use to analyze $\tilde{\mathcal{F}}_{n+1}$. It is proven in \cite[7.1]{Blomquist_iterated_delooping}; a special case is dealt with also in \cite{Ching_Harper_derived_Koszul_duality}. The proposition is a spectral algebra analogue of Dundas's \cite[2.6]{Dundas_relative_K_theory} higher Hurewicz theorem. 

\begin{prop}[Higher $\TQ$-Hurewicz theorem]\label{Higher_TQ_Hurewicz_Theorem}
	Let $k\geq 0$ and $\capX$ be a $W$-cube in $\capO$-algebras that is objectwise cofibrant. If $\capX$ is $(\id + 1)(k + 1)$-cartesian, then so is $\capX\rightarrow UQ\capX$.
\end{prop}

The uniform cartesian estimates given by Proposition \ref{Higher_TQ_Hurewicz_Theorem} applied to $\mathcal{E}_{n+1}, \mathcal{B}_{n+1}$ imply a (slightly weaker) uniform cartesian estimate on $\tilde{\mathcal{F}}_{n+1}$. 
\begin{prop}
	\label{prop:uniform_cartesian_estimates_for_F}
	Let $n\geq -1$. The coface $(n+1)$-cube $\tilde{\mathcal{F}}_{n+1}$ associated to $F\rightarrow\tilde{F}$ is $\id$-cartesian.
\end{prop}
\begin{proof}
	It follows from \ref{Higher_TQ_Hurewicz_Theorem} that both $\mathcal{E}_{n+1}$ and $\mathcal{B}_{n+1}$ are $(n+2)$-cartesian and so \cite[3.8]{Ching_Harper} the cube $\mathcal{E}_{n+1} \to \mathcal{B}_{n+1}$ is $(n+1)$-cartesian. This means that the iterated homotopy fiber \cite[2.6]{Ching_Harper_derived_Koszul_duality} of $\mathcal{E}_{n+1} \to \mathcal{B}_{n+1}$ is $n$-connected. Since this is weakly equivalent to the iterated homotopy fiber of $\tilde{\mathcal{F}}_{n+1}$, we conclude that $\tilde{\mathcal{F}}_{n+1}$ is $(n+1)$-cartesian. Repeating this argument on all subcubes completes the proof.
\end{proof}

The following two short lemmas are used in the proof of Proposition \ref{prop:UQ_preserves_id_cartesian}, which states that levelwise application of the $\TQ$-homology functor preserves this cartesian estimate on $\tilde{\mathcal{F}}_{n+1}$.

\begin{lem}\label{lem:U_preserves_cartesian}
	Let $k \in \mathbb{Z}$ and let $\mathcal{Y}$ be a $W$-cube in $J$-algebras. If $\mathcal{Y}$  is $k$-cartesian, then so is $U\mathcal{Y}$.
\end{lem}
\begin{proof}
	This is because $U$ is a right Quillen functor and preserves connectivity of all maps, since this connectivity is calculated in the underlying category $\ModR$.
\end{proof}

\begin{lem}\label{lem:Q_preserves_cocartesian}
	Let $k \geq -1$ and let $\mathcal{X}$ be an objectwise cofibrant $W$-cube in $\capO$-algebras. If $\mathcal{X}$ is $k$-cocartesian, then so is $Q\mathcal{X}$. 
\end{lem}
\begin{proof}
	If $\abs{W} = 0$ or 1, note that an $\capO$-algebra (resp. a map between $\capO$-algebras) is $k$-cartesian if and only if it is $k$-connected. The result now follows from \cite[1.9(b)]{Harper_Hess} and the observation that if $X$ is cofibrant, then $\TQ(X) \simeq UQ(X)$. To show, more generally, that $Q\mathcal{X}$ is $k$-cocartesian, let $\mathcal{P}_1W$ be the poset of subsets $V \subsetneqq W$. By assumption, $\hocolim_{\mathcal{P}_1W}\mathcal{X} \to \mathcal{X}_W$ is a $k$-connected map of cofibrant objects, so $\hocolim_{\mathcal{P}_1W}Q\mathcal{X} \simeq Q\hocolim_{\mathcal{P}_1W}X \to Q\mathcal{X}_W$ is also $k$-connected, by the first part of the proof.
\end{proof}
\begin{prop}\label{prop:UQ_preserves_id_cartesian}
	Let $\mathcal{X}$ be a $W$-cube in $\capO$-algebras. If $\mathcal{X}$ is objectwise cofibrant and is $\id$-cartesian, then so is $UQ\mathcal{X}$.
\end{prop}

\begin{rem}\label{rem:low_dim_example}
	Before we give the proof of Proposition \ref{prop:UQ_preserves_id_cartesian} in full generality, here is the argument assuming that $\mathcal{X}$ is a 2-cube, i.e., that $W = \sett{1,2}$. In this case, $\mathcal{X}$ is the commutative diagram
	\begin{align*}
	\xymatrix{
		\mathcal{X}_\emptyset \ar[r] \ar[d] & \mathcal{X}_{\{1\}} \ar[d]\\
		\mathcal{X}_{\{2\}} \ar[r] & \mathcal{X}_{\{1,2\}}
	}
	\end{align*}
	
	\noindent
	in $\AlgO$, where each object is $(-1)$-connected (i.e., 0-cartesian as a 0-cube), each map is 1-connected (i.e., 1-cartesian as a 1-cube), and the entire square is 2-cartesian. 
	
	That the objects and maps of $UQ\capX$ are appropriately connected follows as in the proof of Lemma \ref{lem:Q_preserves_cocartesian}. Let us now show that $UQ\mathcal{X}$ is 2-cartesian. The dual Blakers-Massey theorem of Ching-Harper \cite[1.9]{Ching_Harper} implies that $\mathcal{X}$ is $k$-cocartesian, where
	\[
	k = \min\sett{k_{12}+1, k_1+k_2+2} = \min\sett{2+1, 1+1+2}=3
	\] 
	By Lemma \ref{lem:Q_preserves_cocartesian}, this means that $Q\mathcal{X}$ is also 3-cocartesian. It is now important to observe that $Q\mathcal{X}$ is a diagram in the stable category $\AlgJ$, so the fact that it is 3-cocartesian implies it is 2-cartesian; see \cite[3.10]{Ching_Harper}. Hence, by Lemma \ref{lem:U_preserves_cartesian}, $UQ\mathcal{X}$ is also 2-cartesian.
	
	To see that $UQ\mathcal{X}$ is objectwise cofibrant, recall that $Q$ is a left Quillen functor and that \cite[5.49]{Harper_Hess} the functor $U$ preserves cofibrant objects.
\end{rem}

\begin{proof}[Proof of Proposition \ref{prop:UQ_preserves_id_cartesian}]
	Objectwise cofibrancy is proven in the same way as in Remark \ref{rem:low_dim_example}. To show that $UQ\mathcal{X}$ is $\id$-cartesian, we induct on $n$. The cases $\abs{W} = 0, 1,2$ are handled in Remark \ref{rem:low_dim_example}. Suppose now $\mathcal{X}$ is a $W$-cube with $\abs{W} =n \geq 3$ and that the result holds for all $k$-cubes with $k < n$. This verifies that $UQ\mathcal{X}$ is $\id$-cartesian on all strict subcubes, so we must only further show that $UQ\mathcal{X}$ is itself $n$-cartesian. 
	
	As in Remark \ref{rem:low_dim_example}, we first establish a cocartesian estimate on $\mathcal{X}$, but now use the higher dual Blakers-Massey Theorem of Ching-Harper. Adopting the notation of \cite[1.11]{Ching_Harper}, observe that each $k_V$ (the cartesianness of a particular $\abs{V}$-dimensional subcube of $\mathcal{X}$) is equal to $\abs{V}$ by assumption that $\mathcal{X}$ is $\id$-cartesian. It follows that for any partition $\lambda$ of $W$, we have
	\[
	\abs{W} + \sum_{V \in \lambda}k_V = n + \sum_{V \in \lambda}\abs{V} = n + n = 2n
	\]
	On the other hand, 
	\[
	k_W + \abs{W} - 1 = n + n -1 = 2n-1
	\]
	Hence, $\mathcal{X}$ is $(2n-1)$-cocartesian. By Lemma \ref{lem:Q_preserves_cocartesian}, this means $Q\mathcal{X}$ is also $(2n-1)$-cocartesian. Since $Q\mathcal{X}$ is in the stable category $\AlgJ$, the proof of \cite[3.10]{Ching_Harper} implies that $Q\mathcal{X}$ is $(2n-1)-n+1 = n$-cartesian. Therefore, by Lemma \ref{lem:U_preserves_cartesian}, $UQ\mathcal{X}$ is also $n$-cartesian.
\end{proof}

We are now in a position to prove the main result of this section. 
\begin{prop}\label{prop:horizontal_direction}
	Let $n\geq -1$ and $k\geq 0$. The coface $(n+1)$-cube associated to $(UQ)^kF\rightarrow(UQ)^k\tilde{F}$ is $\id$-cartesian. In particular, the natural map $(UQ)^kF \to \holim_\Delta(UQ)^k\tilde{F}$ is a weak equivalence.
\end{prop}

\begin{proof}
	The first part follows inductively from Propositions \ref{prop:uniform_cartesian_estimates_for_F} and \ref{prop:UQ_preserves_id_cartesian}. The second part follows by observing that this cartesian estimate implies that the natural map $(UQ)^kF \to \holim_{\Delta^\leq n}(UQ)^k\tilde{F}$ is $(n+1)$ connected (see Proposition \ref{prop:cofinality}), then using the associated $\lim^1$ short exact sequence.
\end{proof}

\begin{rem}
	The increasing connectivity proven in Proposition \ref{prop:horizontal_direction} implies that, for all $k \geq 0$, the map of towers $\sett{(UQ)^kF}_s \to \sett{\mathsf{Tot}_s(UQ)^k\tilde{F}}_s$ is a pro-$\pi_\ast$ isomorphism.
\end{rem}

\begin{rem}
	If one relaxes the connectivity assumptions on $E, B$, but can still show that $\tilde{\mathcal{F}}_{n+1}$ is $\id$-cartesian, then Proposition \ref{prop:horizontal_direction} remains valid. In this case, since the connectivities of $E, B$ do not play a role in the following section, the conclusion of Theorem \ref{MainTheorem1} also remains valid. We thank the referee for pointing this out.
\end{rem}

\section{Analysis of the vertical direction}
The purpose of this section is to analyze the maps $(**)$ and, in particular, to prove Proposition \ref{prop:vertical_direction}. The basic idea is to first show that, up to homotopy, there is an extra codegeneracy in each coaugmented cosimplicial diagram $\tilde{F}^n \to (UQ)^{\bullet+1}\tilde{F}^n$. This is accomplished by showing that each $\tilde{F}^n$ is weakly equivalent to an $\capO$-algebra of the form $UY$ and observing that the diagram $UY \to (UQ)^{\bullet+1}UY$ has an extra codegeneracy on the nose. A short spectral sequence argument then completes the analysis. 

\begin{lem}\label{lem:Fn_weak_equiv_to_J_alg}
	For each $n \geq 0$, there is a fibrant and cofibrant $J$-algebra $G^n$ with a natural zigzag of weak equivalences $UG^n \simeq \tilde{F}^n$ in $\AlgO$.
\end{lem}
\begin{proof}
	We will prove the $n=0$ case. The proof is essentially the same for $n \geq 1$. By definition, and commuting $U$ past a homotopy limit, we have a natural zigzag of weak equivalences
	\[
	\tilde{F}^0 \simeq \hofib(UQE \to UQB) \simeq U\hofib(QE \to QB)
	\]
	and the lemma follows by letting $G^0$ be the functorial cofibrant replacement of $\hofib(QE \to QB)$ in $\AlgJ$.
\end{proof}

\begin{lem}\label{lem:J_algebras_have_extra_codegeneracies}
	If $Y$ is in $\AlgJ$, the diagram $UY \to (UQ)^{\bullet+1}UY$ has an extra codegeneracy.
\end{lem}
\begin{proof}
	One obtains an extra codegeneracy by defining $s^n = U(QU)^{n+1}Y \overset{U(QU)^n\epsilon}{\to} U(QU)^nY$ for all $n \geq 0$, where $\epsilon$ is the counit associated to the $(Q, U)$ adjunction.
\end{proof}
\begin{lem}\label{lem:extra_codegeneracy}
	If the coaugmented cosimplicial $\capO$-algebra $X^{-1} \to X$ has an extra codegeneracy and $X^{-1}$ is fibrant, then the natural map $\sett{X^{-1}}_s \to \sett{\mathsf{Tot}_sX}_s$ is a pro-$\pi_\ast$ isomorphism.
\end{lem}
\begin{rem}
In the proof below, we use the spectral sequence associated to a tower of fibrations of $\capO$-algebras. For details of the construction, see \cite[8.31]{Ching_Harper_derived_Koszul_duality}. It is essentially the same as the homotopy spectral sequence  \cite[X.6]{Bousfield_Kan} of Bousfield-Kan; see also \cite[VIII.1]{Goerss_Jardine}.
\end{rem}
\begin{proof}
	Fix $n \in \mathbb{Z}$ and consider the coaugmented cosimplicial abelian group $\pi_nX^{-1} \to \pi_nX$. The assumed extra codegeneracy implies that for any $s \geq 0$, we have $\pi^s\pi_nX^{-1} \overset{\iso}{\to} \pi^s\pi_nX$. It follows that there is an induced isomorphism on $E^2$ pages of the homotopy spectral sequences associated to $\sett{X^{-1}}_s$ and $\sett{\mathsf{Tot}_sX}_s$. (Here, we are using the fact that, since $X^{-1}$ is fibrant, the constant cosimplicial diagram with value $X^{-1}$ is Reedy fibrant.) The result now follows from \cite[8.36]{Ching_Harper_derived_Koszul_duality}.
\end{proof}

\begin{prop}\label{prop:vertical_direction}
For each $n\geq 0$, the $\TQ$-completion map $\tilde{F}^n\wequiv ({\tilde{F}^n})^\wedge_\TQ$ is a weak equivalence.
\end{prop}
\begin{proof}
First, note that both $\tilde{F}^n$ and $UG^n$ (as constructed in Lemma \ref{lem:Fn_weak_equiv_to_J_alg}) are cofibrant. By taking further functorial replacements, it follows from Lemma \ref{lem:Fn_weak_equiv_to_J_alg} that there is a natural zigzag of weak equivalences $UG^n \simeq (UG^n)^c \simeq (\tilde{F}^n)^c \simeq \tilde{F}^n$ in which each object is cofibrant. This induces a zigzag of towers
\begin{align*}
\xymatrix@C=-1mm{
\sett{UG^n}_s \ar[d] & \simeq & \sett{\tilde{F}^n}_s \ar[d]^-{(**)}\\
\{\mathsf{Tot}_s(UQ)^{\bullet+1}UG^n\}_s & \simeq & \{\mathsf{Tot}_s(UQ)^{\bullet+1}\tilde{F}^n\}_s
}
\end{align*}
and, by Lemmas \ref{lem:J_algebras_have_extra_codegeneracies} and  \ref{lem:extra_codegeneracy}, the left-hand vertical map is a pro-$\pi_\ast$ isomorphism. It follows that $(**)$ is a pro-$\pi_\ast$ isomorphism as well. Hence, $\tilde{F}^n \simeq (\tilde{F}^n)_\TQ^\wedge$ is a weak equivalence (see Remark \ref{rem:pro_pi_iso_implies_weak_equivalence}). 
\end{proof}

\section{$\TQ$-completion of homotopy pullback squares}
\label{sec:homotopy_pullback}
In this section, we prove Theorem \ref{MainTheorem2}. The strategy of proof is essentially the same as that used in the proof of Theorem \ref{MainTheorem1}. The new arguments given in this section are needed to obtain an analogue of Proposition \ref{prop:uniform_cartesian_estimates_for_F}; this is the content of Proposition \ref{prop:uniform_cartesian_esimates_for_A} below.

As in the proof of Theorem \ref{MainTheorem1}, we may assume that $B, X, Y$ are cofibrant, and we then build the associated cosimplicial resolutions of $B, X, Y$ with respect to $\TQ$-homology; then take levelwise homotopy pullbacks to obtain a coagumented cosimplicial diagram $A \to \tilde{A}$. In other words, we obtain maps of coaugmented cosimplicial $\capO$-algebras of the form
\begin{align}\label{diagram:construction_of_A_tilde}
\xymatrix{
\Big( A\to\tilde{A}\Big) \ar[r] \ar[d] & \Big(X \to (UQ)^{\bullet+1}X\Big) \ar[d]\\
\Big( B \to (UQ)^{\bullet+1}B \Big) \ar[r] & \Big(Y \to (UQ)^{\bullet+1}Y\Big)
}
\end{align}
such that on each fixed cosimplicial degree, one has a homotopy pullback diagram. For instance, in cosimplicial degrees 0, 1 we have homotopy pullback diagrams of the form
\begin{align}
\xymatrix{
	\tilde{A}^0 \ar[r] \ar[d] & (UQ)X \ar[d]\\
	(UQ)B \ar[r] & (UQ)Y	
} 
\quad
\xymatrix{
	\tilde{A}^1 \ar[r] \ar[d] & (UQ)^2X \ar[d]\\
	(UQ)^2B \ar[r] & (UQ)^2Y	
}
\end{align}
in $\AlgO$, and these are coaugmented by the diagram in Theorem \ref{MainTheorem2}. For the same reasons as in the proof of Theorem \ref{MainTheorem1}, we may assume $A \to \tilde{A}$ is objectwise cofibrant, and that $\tilde{A}$ is a Reedy fibrant cosimplicial $\capO$-algebra.

\begin{proof}[Proof of Theorem \ref{MainTheorem2}]
Construct a diagram identical to \eqref{eq:resolution_of_F_diagram}. For the same reasons as in Theorem \ref{MainTheorem1}, the maps $(**)$ induces pro-$\pi_\ast$ isomorphisms after applying $\mathsf{Tot}_s$. The result now follows from Proposition \ref{prop:uniform_cartesian_esimates_for_A} below and arguing as in the proof of Theorem \ref{MainTheorem1}.
\end{proof}

\begin{prop}[cf. Proposition \ref{prop:uniform_cartesian_estimates_for_F}]\label{prop:uniform_cartesian_esimates_for_A}
	Let $n \geq - 1$. The coface $(n+1)$-cube associated to $A \to \tilde{A}$ is $\id$-cartesian.
\end{prop}

\begin{rem}\label{rem:low_dim_example_for_cartesianness_of_A}
	As in the case of Proposition \ref{prop:uniform_cartesian_estimates_for_F}, it may be helpful to first understand a low-dimensional example of Proposition \ref{prop:uniform_cartesian_esimates_for_A} before attacking the proof in full generality. Suppose we wish to show the 1-cube $A \to \tilde{A}^0$ is $\id$-cartesian. Consider the corresponding 1-cubes of $A, X, Y$ to obtain a commutative diagram of the form
	\begin{align}\label{diagram:low_dim_example_for_cartesianness_of_A}
	\xymatrix@!0{
	A \ar[dd] \ar[dr] \ar[rr] && X \ar[dd]|\hole \ar[dr]\\
	& \tilde{A}^0 \ar[dd] \ar[rr] &&UQX \ar[dd]\\
	B \ar[rr]|\hole \ar[dr]&& Y \ar[dr]\\
	& UQB \ar[rr] && UQY
	}
	\end{align}
in $\AlgO$. We will make frequent use of \cite[3.8]{Ching_Harper} in the following analysis. 

Since the back and front faces of \eqref{diagram:low_dim_example_for_cartesianness_of_A} are both homotopy pullback diagrams (i.e., infinitely caretesian), the entire 3-cube is also infinitely cartesian. The 1-cubes $X \to UQX$ and $Y \to UQY$ are both 2-cartesian (i.e., the maps are 2-connected) by Proposition \ref{Higher_TQ_Hurewicz_Theorem} and so the right-hand face of \eqref{diagram:low_dim_example_for_cartesianness_of_A} is 1-cartesian. Therefore, the left-hand face is 1-cartesian as well. Since $B \to UQB$ is 2-cartesian (also by Proposition \ref{Higher_TQ_Hurewicz_Theorem}), we conclude that $A \to \tilde{A}^0$ is 1-cartesian. One then repeats this argument on all subcubes of $A \to \tilde{A}^0$, i.e., on the objects $A$ and $\tilde{A}^0$.
\end{rem}

\begin{proof}[Proof of Proposition \ref{prop:uniform_cartesian_esimates_for_A}]
Denote by $\tilde{\mathcal{A}}_{n+1}, \mathcal{B}_{n+1}, \mathcal{X}_{n+1}, \mathcal{Y}_{n+1}$ the coface $(n+1)$ cubes associated to the coaugmented cosimplicial diagrams in \eqref{diagram:construction_of_A_tilde}. Let $\mathcal{C}$ be any subcube of $\mathcal{\tilde{A}}_{n+1}$, say of dimension $k$. Let $\mathcal{C}_B, \mathcal{C}_X, \mathcal{C}_Y$ denote the corresponding subcubes and consider the commutative diagram of subcubes
\begin{align}\label{diagram:subcubes_for_proof_of_uniform_cart_of_A}
\xymatrix{
\mathcal{C} \ar[r] \ar[d] & \mathcal{C}_X \ar[d]\\
\mathcal{C}_B \ar[r] & \mathcal{C}_Y	
}
\end{align}
As in Remark \ref{rem:low_dim_example_for_cartesianness_of_A}, the first step is to establish that the cube \eqref{diagram:subcubes_for_proof_of_uniform_cart_of_A} is infinitely cartesian. This is accomplished by Lemma \ref{lem:infintiely_cart_for_uniform_cart_of_A_proof} below. Next, Proposition \ref{Higher_TQ_Hurewicz_Theorem} implies that $\mathcal{C}_X$ and $\mathcal{C}_Y$ are both $(k+1)$-caratesian, so the cube $\mathcal{C}_X \to \mathcal{C}_Y$ is $k$-cartesian. Therefore, the cube $\mathcal{C}\to \mathcal{C}_B$ is $k$-cartesian. Since $\mathcal{C}_B$ is $(k+1)$-cartesian (also by Proposition \ref{Higher_TQ_Hurewicz_Theorem}), we conclude that $\mathcal{C}$ is $k$-cartesian.
\end{proof}

\begin{lem}\label{lem:infintiely_cart_for_uniform_cart_of_A_proof}
	For any subcube $\mathcal{C}$ of $\tilde{\mathcal{A}}_{n+1}$, the cube constructed in \eqref{diagram:subcubes_for_proof_of_uniform_cart_of_A} is infinitely cartesian.
\end{lem}
\begin{proof}
	The proof is by induction on $k$. If $k=0$, then $\mathcal{C}$ is a single object and the lemma follows by construction. If $k \geq 1$, we may write $\mathcal{C}$ as a map of $(k-1)$-dimensional subcubes $\mathcal{D} \to \mathcal{E}$ of $\tilde{\mathcal{A}}_{n+1}$. Consider the commutative diagram of subcubes
	\begin{align*}
	\xymatrix@!0{
		\mathcal{D} \ar[dd] \ar[dr] \ar[rr] && \mathcal{D}_X \ar[dd]|\hole \ar[dr]\\
		& \mathcal{E} \ar[dd] \ar[rr] &&\mathcal{E}_X \ar[dd]\\
		\mathcal{D}_B \ar[rr]|\hole \ar[dr]&& \mathcal{D}_Y \ar[dr]\\
		& \mathcal{E}_B \ar[rr] && \mathcal{E}_Y
	}
	\end{align*}	
and note that this diagram is precisely \eqref{diagram:subcubes_for_proof_of_uniform_cart_of_A}. By induction, the back and front faces (which are themselves both $(k+1)$-cubes) are both infinitely cartesian, so the whole cube is as well.
\end{proof}

\section{Appendix: cubical diagrams}\label{sec:appendix}
The purpose of this appendix is to briefly summarize the tools of cubical homotopy theory used in this paper, particularly in Section \ref{sec:horizontal_direction}. While these notions can be defined in other settings, we have phrased them in the context of $\capO$-algebras to keep this section appropriately focused. For the more interested reader, useful references for cubical diagrams of spaces include \cite[A.8]{Dundas_Goodwillie_McCarthy}, \cite{Goodwillie_calculus_2}, and \cite{Munson_Volic_book_project}. In the context of $\capO$-algebras, see \cite{Ching_Harper} and \cite{Ching_Harper_derived_Koszul_duality}.

\begin{defn}
	Let $\mathcal{X}$ be a $W$-cube of $\capO$-algebras indexed on the set $W= [n]$, where $[n]: =\sett{1, 2, \ldots, n}$. Let $\mathcal{P}_0([n])$ be the poset of nonempty subsets of $[n]$. We say that $\mathcal{X}$ is \emph{$k$-cartesian} if the natural map $\mathcal{X}_\emptyset \to \holim_{\mathcal{P}_0([n])}\mathcal{X}$ is $k$-connected.
\end{defn}

The connectivity of $\mathcal{X}_\emptyset \to \holim_{\mathcal{P}_0([n])}\mathcal{X}$ gives information about the cube $\mathcal{X}$ as a whole. One might also be interested in subcubes (see \cite[3.6]{Blomquist_Harper_integral_chains} or \cite[A.8.0.1]{Dundas_Goodwillie_McCarthy}) of $\mathcal{X}$. This motivates the following definition, which appears in \cite{Dundas_relative_K_theory} and \cite{Dundas_Goodwillie_McCarthy}.

\begin{defn}
	Given a function $f \colon \mathbb{N} \to \mathbb{Z}$, we say that a cube $\mathcal{X}$ of $\capO$-algebras is \emph{$f$-cartesian} if each $d$-dimensional subcube of $\mathcal{X}$ is $f(d)$-cartesian; here, $\mathbb{N}$ denotes the non-negative integers.
\end{defn}
\begin{rem}
	For instance, to say that a cube $\mathcal{X}$ is $\id$-cartesian means that each $d$-dimensional subcube of $\mathcal{X}$ is $d$-cartesian.
\end{rem}

\begin{defn}\label{general_coface_cubes}
	Let $Z^{-1} \overset{d^0}{\to} Z$ be a coaugmented cosimplicial $\capO$-algebra. The \emph{coface $(n+1)$-cube $\mathcal{X}_{n+1}$ associated to $Z^{-1} \to Z$} is the canonical $(n+1)$-cube constructed using the cosimplicial relations $d^jd^i = d^id^{j-1}$ for $i < j$.
\end{defn}
\begin{rem}
	For instance, $\mathcal{X}_2$ has the form on the left, and $\mathcal{X}_3$ the form on the right.
\end{rem}

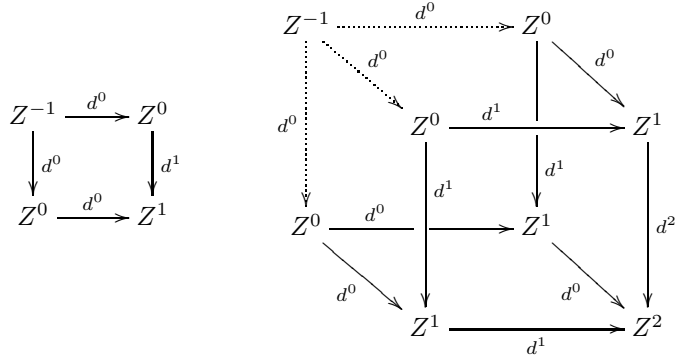
\begin{align*}
\xymatrix{
	\\
	Z^{-1} \ar^-{d^0}[r] \ar^-{d^0}[d] & Z^0 \ar^-{d^1}[d]\\
	Z^0 \ar^-{d^0}[r] & Z^1
}\quad \quad \quad
\xymatrix{
Z^{-1} \ar@{.>}_-{d^0}[dd] \ar@{.>}^-{d^0}[dr] \ar@{.>}^-{d^0}[rr] && Z^0 \ar^(.7){d^1}[dd]|\hole \ar^-{d^0}[dr]\\
& Z^0 \ar^(.3){d^1}[dd] \ar^(.3){d^1}[rr] &&Z^1 \ar^-{d^2 }[dd]\\
Z^0 \ar^(.3){d^0}[rr]|\hole \ar_-{d^0}[dr]&& Z^1 \ar_-{d^0}[dr]\\
& Z^1\ar_-{d^1}[rr] && Z^2
}
\end{align*}

One of the reasons cubical diagrams are useful is that (as described above) they naturally arise from cosimplicial diagrams, combined with the following fact, which is proven in \cite[Section 6]{Carlsson}, \cite{Dugger_homotopy_colimits}, and \cite[6.7]{Sinha_cosimplicial_models}.

\begin{prop}\label{prop:cofinality}
	For $n \geq 0$, the composite
	\begin{align*}
	\mathcal{P}_0([n]) \to \Delta^{\leq n}
	\end{align*}
	is left cofinal (i.e., homotopy initial).
\end{prop}

The upshot is this: Given a coaugmented cosimplicial $\capO$-algebra $Z^{-1}\to Z$, the map $Z^{-1} \to \holim_{\Delta^{\leq n}}Z$ is $k$-connected if and only if the associated coface $(n+1)$-cube $\mathcal{X}_{n+1}$ is $k$-cartesian.

\bibliographystyle{plain}
\bibliography{FibrationSequences}

\begin{thebibliography}{10}

\bibitem{Basterra}
M.~Basterra.
\newblock Andr\'e-{Q}uillen cohomology of commutative {$S$}-algebras.
\newblock {\em J. Pure Appl. Algebra}, 144(2):111--143, 1999.

\bibitem{Basterra_Mandell}
M.~Basterra and M.~A. Mandell.
\newblock Homology and cohomology of {$E\sb \infty$} ring spectra.
\newblock {\em Math. Z.}, 249(4):903--944, 2005.

\bibitem{Blomquist_iterated_delooping}
J~R.. {Blomquist}.
\newblock {Iterated delooping and desuspension of structured ring spectra}.
\newblock {\em arXiv e-prints}, page arXiv:1910.12442, Oct 2019.

\bibitem{Blomquist_Harper_integral_chains}
Jacobson~R. Blomquist and John~E. Harper.
\newblock Integral chains and {B}ousfield-{K}an completion.
\newblock {\em Homology Homotopy Appl.}, 21(2):29--58, 2019.

\bibitem{Bousfield_Kan}
A.~K. Bousfield and D.~M. Kan.
\newblock {\em Homotopy limits, completions and localizations}.
\newblock Lecture Notes in Mathematics, Vol. 304. Springer-Verlag, Berlin,
  1972.

\bibitem{Carlsson}
G.~Carlsson.
\newblock Derived completions in stable homotopy theory.
\newblock {\em J. Pure Appl. Algebra}, 212(3):550--577, 2008.

\bibitem{Ching_Harper}
M.~Ching and J.~E. Harper.
\newblock Higher homotopy excision and {B}lakers-{M}assey theorems for
  structured ring spectra.
\newblock {\em Adv. Math.}, 298:654--692, 2016.

\bibitem{Ching_Harper_derived_Koszul_duality}
M.~Ching and J.~E. Harper.
\newblock Derived {K}oszul duality and {TQ}-homology completion of structured
  ring spectra.
\newblock {\em Adv. Math.}, 341:118--187, 2019.

\bibitem{Dugger_homotopy_colimits}
D.~Dugger.
\newblock A primer on homotopy colmits.
\newblock {\em Preprint}, 2008.
\newblock Available at\\ \verb=http://pages.uoregon.edu/ddugger/=.

\bibitem{Dundas_relative_K_theory}
B.~I. Dundas.
\newblock Relative {$K$}-theory and topological cyclic homology.
\newblock {\em Acta Math.}, 179(2):223--242, 1997.

\bibitem{Dundas_Goodwillie_McCarthy}
B.~I. Dundas, T.~G. Goodwillie, and R.~McCarthy.
\newblock {\em The local structure of algebraic {K}-theory}, volume~18 of {\em
  Algebra and Applications}.
\newblock Springer-Verlag London, Ltd., London, 2013.

\bibitem{Dwyer_exotic_convergence}
William~G. Dwyer.
\newblock Exotic convergence of the {E}ilenberg-{M}oore spectral sequence.
\newblock {\em Illinois J. Math.}, 19(4):607--617, 1975.

\bibitem{EKMM}
A.~D. Elmendorf, I.~Kriz, M.~A. Mandell, and J.~P. May.
\newblock {\em Rings, modules, and algebras in stable homotopy theory},
  volume~47 of {\em Mathematical Surveys and Monographs}.
\newblock American Mathematical Society, Providence, RI, 1997.
\newblock With an appendix by M. Cole.

\bibitem{Goerss_Jardine}
P.~G. Goerss and J.~F. Jardine.
\newblock {\em Simplicial homotopy theory}, volume 174 of {\em Progress in
  Mathematics}.
\newblock Birkh\"auser Verlag, Basel, 1999.

\bibitem{Goodwillie_calculus_2}
T.~G. Goodwillie.
\newblock Calculus. {II}. {A}nalytic functors.
\newblock {\em $K$-Theory}, 5(4):295--332, 1991/92.

\bibitem{Harper_symmetric_spectra}
J.~E. Harper.
\newblock Homotopy theory of modules over operads in symmetric spectra.
\newblock {\em Algebr. Geom. Topol.}, 9(3):1637--1680, 2009.
\newblock Corrigendum: \emph{Algebr. Geom. Topol.}, 15(2):1229--1237, 2015.

\bibitem{Harper_bar_constructions}
J.~E. Harper.
\newblock Bar constructions and {Q}uillen homology of modules over operads.
\newblock {\em Algebr. Geom. Topol.}, 10(1):87--136, 2010.

\bibitem{Harper_Hess}
J.~E. Harper and K.~Hess.
\newblock Homotopy completion and topological {Q}uillen homology of structured
  ring spectra.
\newblock {\em Geom. Topol.}, 17(3):1325--1416, 2013.

\bibitem{Harper_Zhang}
J.~E. Harper and Y.~Zhang.
\newblock Topological {Q}uillen localization of structured ring spectra.
\newblock {\em Tbilisi Math. J.}, 12:67--89, 2019.

\bibitem{Hovey_Shipley_Smith}
M.~Hovey, B.~Shipley, and J.~H. Smith.
\newblock Symmetric spectra.
\newblock {\em J. Amer. Math. Soc.}, 13(1):149--208, 2000.

\bibitem{Lawson}
T.~Lawson.
\newblock The plus-construction, {B}ousfield localization, and derived
  completion.
\newblock {\em J. Pure Appl. Algebra}, 214(5):596--604, 2010.

\bibitem{Munson_Volic_book_project}
B.~A. Munson and I.~Voli\'c.
\newblock {\em Cubical homotopy theory}, volume~25 of {\em New Mathematical
  Monographs}.
\newblock Cambridge University Press, Cambridge, 2015.

\bibitem{Schwede_homotopy_groups}
S.~Schwede.
\newblock On the homotopy groups of symmetric spectra.
\newblock {\em Geom. Topol.}, 12(3):1313--1344, 2008.

\bibitem{Schwede_book_project}
S.~Schwede.
\newblock Symmetric spectra.
\newblock {\em Available at:\\
  \verb=http://www.math.uni-bonn.de/people/schwede/SymSpec-v3.pdf=}, 2012.

\bibitem{Shipley_commutative_ring_spectra}
B.~Shipley.
\newblock A convenient model category for commutative ring spectra.
\newblock In {\em Homotopy theory: relations with algebraic geometry, group
  cohomology, and algebraic $K$-theory}, volume 346 of {\em Contemp. Math.},
  pages 473--483. Amer. Math. Soc., Providence, RI, 2004.

\bibitem{Sinha_cosimplicial_models}
D.~P. Sinha.
\newblock The topology of spaces of knots: cosimplicial models.
\newblock {\em Amer. J. Math.}, 131(4):945--980, 2009.

\end{thebibliography}

\end{document}